\documentclass[10pt]{article}
\usepackage{amssymb}
\usepackage{color}
\usepackage{amsthm}
\usepackage{amsmath}
\usepackage{amsfonts}
\usepackage{amstext}
\usepackage[latin1]{inputenc}
\usepackage{amscd}
\usepackage{latexsym}
\usepackage{bm}
\usepackage{amssymb}
\usepackage[all]{xy}
\usepackage{euscript}
\usepackage{a4wide}
\usepackage{amsmath,graphicx}
\usepackage{mathrsfs,mathtools}
\usepackage[OT4]{fontenc}

\newtheorem{theorem}{Theorem}

\newtheorem{property}{Property}
\newtheorem{lemma}{Lemma}

\newtheorem{definition}{Definition}

\newcommand{\N}{\mathbb{N}}
\newcommand{\R}{\mathbb{R}}

\newcommand{\Il}{I^{\alpha}_{0,+}}

\newcommand{\Dl}{{^{c}}D^{\alpha}_{0,+}}

\newcommand{\seq}{y_m}

\newcommand{\fonction}[5]{\begin{array}[t]{lrcl}#1 :&#2 &\longrightarrow &#3\\&#4& \longmapsto &#5 \end{array}}

\begin{document}

\title{Variational Methods for the Fractional \\
Sturm--Liouville Problem}

% ------------------------

\author{Ma\l gorzata Klimek$^1$\\
\texttt{mklimek@im.pcz.pl}
\and
Tatiana Odzijewicz$^2$\\
\texttt{tatianao@ua.pt}
\and
Agnieszka B. Malinowska$^3$\\
\texttt{a.malinowska@pb.edu.pl}}

% --------------------------

\date{$^1$Institute of Mathematics, Czestochowa University of Technology\\
42-200 Cz\k{e}stochowa, Poland\\[0.3cm]
$^2$Center for Research and Development in Mathematics and Applications\\
Department of Mathematics, University of Aveiro, 3810-193 Aveiro, Portugal\\[0.3cm]
$^3$Faculty of Computer Science, Bialystok University of Technology\\
15-351 Bia\l ystok, Poland}

% ------------------------

\maketitle

% ------------------------

\begin{abstract}
This article is devoted to the regular fractional Sturm--Liouville eigenvalue problem. Applying methods of fractional variational analysis we prove existence of countable set of orthogonal solutions and corresponding eigenvalues. Moreover, we formulate two results showing that the lowest eigenvalue is the minimum value for a certain variational functional.
\bigskip


\bigskip


\end{abstract}

% ------------------------

\section*{Introduction}

In 1836--1837 French mathematicians Sturm (1803-1853) and Liouville (1809-1855) published series of articles initiating the new subtopic of mathematical analysis - the Sturm--Liouville theory. It deals with the general linear, homogeneous second-order ordinary differential equation of the form
\begin{equation}\label{eq:SLI}
\frac{d}{dx}\left(p(x)\frac{dy}{dx}\right)+q(x)y=\lambda w(x)y,
\end{equation}
where $x\in[a,b]$ and in any particular problem functions $p(x)$, $q(x)$ and $w(x)$ are known. In addition, certain boundary conditions are attached to equation \eqref{eq:SLI}. For specific choices of the boundary  conditions, nontrivial solutions of \eqref{eq:SLI} exist only for particular values of the parameter $\lambda=\lambda^{(m)}$, $m=1,2,\dots$. Constants $\lambda^{(m)}$ are called eigenvalues and corresponding solutions $y^{(m)}(x)$ are called eigenfunctions. For a deeper discussion of the classical Sturm--Liouville theory we refer the reader to \cite{book:GF,sagan,book:Brunt}.

Recently many researchers have focused their attention on certain generalization of Sturm--Liouville problem. Namely, they are interested in equations of the type \eqref{eq:SLI}, however with fractional differential operators (see e.g., \cite{AlM09,QMQ,Malgorzata2,Kli13,Kli13a,Lin,Leo11,Nea09,Ovi12,QCH}) which are integrals and derivatives of arbitrary real or complex order. As mentioned in \cite{isi},
Science Watch of Thomson Reuters identified the fractional calculus as an Emerging Research Front area. This is due to its many applications in science and  engineering. A comprehensive study of the fractional calculus and its applications can be found in several recent books \cite{book:Kilbas,book:Kiryakova,book:Malgorzata,book:Podlubny}. In this note we develop the fractional Sturm--Liouville theory by studying the Sturm--Liouville eigenvalue problem with Caputo fractional derivatives. We shall show that fractional variational principles are useful for the approximation of eigenvalues and eigenfunctions.
Traditional Sturm--Liouville theory does not depend upon the calculus of variations, but
stems from the theory of ordinary linear differential equations. However
the Sturm--Liouville eigenvalue problem is readily formulated as a constrained
variational principle, and this formulation can be used to approximate the
solutions. We emphasize that it has a special importance for the fractional Sturm--Liouville equation since fractional operators are nonlocal and it can be extremely challenging to find analytical solutions to fractional differential equations. Besides allowing convenient approximations many general properties of the eigenvalues can be derived using the variational principle. 

The paper is organized as follows. In Section~\ref{sec:prel} a brief review of the fractional calculus is given and three technical lemmas are shown. Our main results are then formulated and proved in Section~\ref{sec:mr}: we show the existence of orthogonal solutions to the fractional Sturm--Liouville eigenvalue problem (Theorem~\ref{thm:exist}), prove that lowest eigenvalues can be characterized as minimum values of certain functionals (Theorem~\ref{thm:FE} and Theorem~\ref{thm:RQ}). In Subsection~\ref{Example} we illustrate our results through an example. Finally in Appendix, reader can find two lemmas concerning certain convergence properties of fractional and classical derivatives, that play an important role in the proof of Theorem~\ref{thm:exist}.
%---------------------------------------------

\section{Preliminaries}
\label{sec:prel}
The reference books for the fractional calculus are \cite{book:Kilbas,book:Kiryakova,book:Malgorzata,book:Podlubny}. Here we only recall necessary definitions and properties of fractional operators. Moreover, we prove lemmas that will be used in the proof of the main result, Theorem~\ref{thm:exist}.

\noindent

\begin{definition}[Left and right Riemann--Liouville fractional integrals]
We define the left and ~the right Riemann--Liouville fractional integrals
$\textsl{I}_{a+}^{\alpha}$ and $\textsl{I}_{b-}^{\alpha}$ of order $\alpha\in\R$ ($\alpha >0$) by
\begin{equation}\label{eq:def:lRLI}
\textsl{I}_{a+}^{\alpha} f(x):=\frac{1}{\Gamma(\alpha)}\int_a^x \frac{f(t)dt}{(x-t)^{1-\alpha}},~~x\in (a,b],
\end{equation}
and
\begin{equation}\label{eq:def:rRLI}
\textsl{I}_{b-}^{\alpha} f(x):=\frac{1}{\Gamma(\alpha)}\int_x^b \frac{f(x)dt}{(t-x)^{1-\alpha}},~~x\in [a,b),
\end{equation}
respectively. Here $\Gamma(\alpha)$ denotes the Euler's gamma function.
\end{definition}

The following assertion shows that Riemann--Liouville fractional integrals satisfy semigroup property.

\begin{property}[cf. Lemma~2.3 \cite{book:Kilbas}]
Let $\alpha,\beta>0$ and $f\in L^p(a,b)$, ($1\leq p\leq\infty$). Then, equations
\begin{equation*}
\textsl{I}_{a+}^{\alpha}\textsl{I}_{a+}^{\beta} f(x)=\textsl{I}_{a+}^{\alpha+\beta} f(x),
\end{equation*}
and
\begin{equation*}
\textsl{I}_{b-}^{\alpha}\textsl{I}_{b-}^{\beta} f(x)=\textsl{I}_{b-}^{\alpha+\beta} f(x)
\end{equation*}
are satisfied.
\end{property}

\begin{definition}[Left and right Riemann--Liouville fractional derivatives]
The left Riemann--Liouville fractional derivative
of order $\alpha\in\R$ ($0<\alpha <1$) of a function $f$, denoted by $\textsl{D}_{a+}^{\alpha} f$,
is defined by
\begin{equation*}
\forall x\in(a,b],~~\textsl{D}_{a+}^{\alpha} f(x)
:= D\textsl{I}_{a+}^{1-\alpha}f(x).
\end{equation*}
Similarly, the right Riemann--Liouville fractional derivative of order $\alpha$
of a function $f$, denoted by $\textsl{D}_{b-}^{\alpha} f$,
is defined by
\begin{equation*}
\forall x\in[a,b),~~\textsl{D}_{b-}^{\alpha} f(x)
:= -D\textsl{I}_{b-}^{1-\alpha}f(x),
\end{equation*}
where $D=\frac{d}{dx}$.
\end{definition}

As one can see below Riemann--Liouville fractional integral and differential operators of power functions return power functions.

\begin{property}[cf. Property~2.1 \cite{book:Kilbas}]
Now, let $\alpha,\beta>0$, then the following identities hold:
\begin{equation*}
\textsl{I}_{a+}^{\alpha}(t-a)^{\beta-1}(x)=\frac{\Gamma(\beta)}{\Gamma(\beta+\alpha)}(x-a)^{\beta+\alpha-1},
\end{equation*}
\begin{equation*}
\textsl{D}_{a+}^{\alpha}(t-a)^{\beta-1}(x)=\frac{\Gamma(\beta)}{\Gamma(\beta-\alpha)}(x-a)^{\beta-\alpha-1},
\end{equation*}
\begin{equation*}
\textsl{I}_{b-}^{\alpha}(b-t)^{\beta-1}(x)=\frac{\Gamma(\beta)}{\Gamma(\beta+\alpha)}(b-x)^{\beta+\alpha-1},
\end{equation*}
and
\begin{equation*}
\textsl{D}_{b-}^{\alpha}(b-t)^{\beta-1}(x)=\frac{\Gamma(\beta)}{\Gamma(\beta-\alpha)}(b-x)^{\beta-\alpha-1}.
\end{equation*}
\end{property}

\begin{definition}[Left and right Caputo fractional derivatives]
The left and the right Caputo fractional derivatives of order $\alpha\in\R$ ($0<\alpha <1$) are given by
\begin{equation}\label{eq:1}
\forall x \in (a,b], ~~{^{c}}D^{\alpha}_{a+} f(x):=\textsl{D}_{a+}^{\alpha}\left [ f(x)-f(a)\right],
\end{equation}
\begin{equation}\label{eq:1r}
\forall x \in [a,b),~~{^{c}}D^{\alpha}_{b-} f(x):=\textsl{D}_{b-}^{\alpha} \left [f(x)-f(b)\right].
\end{equation}
\end{definition}
Let $0<\alpha<1$ and $f\in AC[a,b]$, then  Caputo fractional derivatives satisfy the following relations:
\begin{equation*}
{^{c}}D^{\alpha}_{a+} f(x)=\textsl{I}_{a+}^{1-\alpha} Df(x)
\end{equation*}
and
\begin{equation*}
{^{c}}D^{\alpha}_{b-} f(x)=-\textsl{I}_{b-}^{1-\alpha} Df(x),
\end{equation*}
respectively.

\begin{property}[cf. Lemma~2.4 \cite{book:Kilbas}]
If $\alpha>0$ and $f\in L^p(a,b)$, ($1\leq p\leq\infty$), then the following is true:
\begin{equation*}
\textsl{D}_{a+}^{\alpha}\textsl{I}_{a+}^{\alpha} f(x)=f(x),
\end{equation*}
\begin{equation*}
\textsl{D}_{b-}^{\alpha}\textsl{I}_{b-}^{\alpha} f(x)=f(x),
\end{equation*}
for almost all $x\in [a,b]$. If function $f$ is continuous, then composition rules hold for all $x\in [a,b]$.
\end{property}

The above property shows that the Riemann--Liouville derivative is the left inverse of the Riemann--Liouville integral, but we cannot claim that it is the right inverse. More precisely, for $1>\alpha>0$ we have the following situation.

\begin{property}[cf. Lemma~2.5 and Lemma~2.6, \cite{book:Kilbas}]
If $f\in L^1(a,b)$ and $\textsl{I}_{a+}^{1-\alpha}f,\textsl{I}_{b-}^{1-\alpha}f\in AC[a,b]$, then the following is true:
\begin{equation}\label{eq:2}
\textsl{I}_{a+}^{\alpha}\textsl{D}_{a+}^{\alpha} f(x)=f(x)-\frac{(x-a)^{\alpha-1}}{\Gamma(\alpha)}\textsl{I}_{a+}^{1-\alpha}(a),
\end{equation}
\begin{equation}\label{eq:2r}
\textsl{I}_{b-}^{\alpha} \textsl{D}_{b-}^{\alpha} f(x)=f(x)-\frac{(b-x)^{\alpha-1}}{\Gamma(\alpha)}\textsl{I}_{b-}^{1-\alpha}(b).
\end{equation}
\end{property}

Next results show that for certain classes of functions Caputo fractional derivatives are inverse operators of Riemann--Liouville fractional integrals.

\begin{property}[cf. Lemma~2.21 \cite{book:Kilbas}]
\label{prop:4}
Let $\alpha>0$ and $\alpha\in\N$ or $\alpha\notin\N$. If $f$ is continuous on the interval $[a,b]$, then
\begin{equation*}
{^{c}}D^{\alpha}_{a+}\textsl{I}_{a+}^{\alpha} f(x)=f(x),
\end{equation*}
\begin{equation*}
{^{c}}D^{\alpha}_{b-}\textsl{I}_{b-}^{\alpha} f(x)=f(x).
\end{equation*}
\end{property}

\begin{property}[cf. Lemma~2.22 \cite{book:Kilbas}]
%\label{prop:4}
Let $1\geq\alpha>0$. If $f\in AC[a,b]$, then
\begin{equation}\label{eq:3}
\textsl{I}_{a+}^{\alpha}{^{c}}D^{\alpha}_{a+} f(x)=f(x)-f(a),
\end{equation}
\begin{equation}\label{eq:3r}
\textsl{I}_{b-}^{\alpha}{^{c}}D^{\alpha}_{b-} f(x)=f(x)-f(b).
\end{equation}
\end{property}

Note that, if $f(a)=0$, then we have
\begin{equation}%\label{eq:3}
\textsl{I}_{a+}^{\alpha}{^{c}}D^{\alpha}_{a+} f(x)=\textsl{I}_{a+}^{1}\textsl{D} f(x)=f(x)-f(a)=f(x),
\end{equation}
and similarly for  $f(b)=0$ we obtain
\begin{equation}%\label{eq:3r}
\textsl{I}_{b-}^{\alpha}{^{c}}D^{\alpha}_{b-} f(x)=-\textsl{I}_{b-}^{1}\textsl{D} f(x)=f(x)-f(b)=f(x).
\end{equation}

For $p$-Lebesgue integrable functions Riemann--Liouville fractional integrals and derivatives satisfy the following composition properties.

\begin{property}[cf. Property~2.2 \cite{book:Kilbas}]
\label{prop:5}
Let $\alpha>\beta>0$ and $f\in L^p(a,b)$, ($1\leq p\leq\infty$). Then, relations
\begin{equation*}
\textsl{D}_{a+}^{\beta}\textsl{I}_{a+}^{\alpha} f(x)=\textsl{I}_{a+}^{\alpha-\beta} f(x),
\end{equation*}
and
\begin{equation*}
\textsl{D}_{b-}^{\beta}\textsl{I}_{b-}^{\alpha} f(x)=\textsl{I}_{b-}^{\alpha-\beta} f(x)
\end{equation*}
are satisfied for almost all $x\in [a,b]$ (in the case of continuous $f$ they are satisfied for all $x\in [a,b]$). In particular, when $\beta=k\in\N$ and $\alpha>k$, then
\begin{equation*}
D^{k}\textsl{I}_{a+}^{\alpha} f(x)=\textsl{I}_{a+}^{\alpha-k} f(x),
\end{equation*}
and
\begin{equation*}
D^{k}\textsl{I}_{b-}^{\alpha} f(x)=(-1)^k\;\textsl{I}_{b-}^{\alpha-k} f(x).
\end{equation*}
\end{property}

In classical calculus, integration by parts formula relates the integral of a product of functions to the integral of their derivative and antiderivative. As we can see below, this formula works also for fractional derivatives, however it changes the type of differentiation: left Riemann--Liouville fractional derivatives are transformed to right Caputo fractional derivatives.

\begin{property}[cf. Lemma~2.19 \cite{book:Malgorzata}]
Assume that $0<\alpha<1$, $f\in AC[a,b]$ and $g\in L^p(a,b)$, ($1\leq p\leq\infty$). Then, the following integration by parts formula holds
\begin{equation}\label{eq:IBP}
\int_a^b f(x)\textsl{D}_{a+}^{\alpha}g(x)\;dx=\int_a^b g(x){^{c}}D^{\alpha}_{b-}f(x)\;dx+\left.f(x)\textsl{I}_{a+}^{1-\alpha}g(x)\right|_{x=a}^{x=b}.
\end{equation}
\end{property}

Finally, let us recall the following property yielding boundedness of the Riemann--Liouville fractional integral in the $L^p(a,b)$ space (cf. Lemma~2.1 \cite{book:Kilbas}).

\begin{property}
Let $\beta \in \mathbb{R}_{+}$ and $p\geq 1$. The fractional integral operator $I^{\beta}_{a,+}$ is bounded in $L^{p}(a,b)$:
\begin{equation}\label{K}
||I^{\beta}_{a+}f||_{L^{p}}\leq K_{\beta}||f||_{L^{p}}\quad K_{\beta}= \frac{(b-a)^{\beta}}{\Gamma(\beta+1)}.
\end{equation}
\end{property}

Now, we are in position to prove lemmas that will be used in the proof of Theorem~\ref{thm:exist}.

\begin{lemma}\label{lem:A}
Let $\alpha\in (0,1) $ and function $\gamma\in C[a,b]$. If
\begin{equation*}
\int_{a}^{b}\gamma(x)D{}^{c}D^{\alpha}_{a+}h(x)dx=0
\end{equation*}
for each $ h\in  C^{1}[a,b]$ such that $D{}^{c}D^{\alpha}_{a+}h\in C[a,b]$ and
fulfilling boundary conditions
\begin{eqnarray}
&&h(a)=I^{1-\alpha}_{a+}h(b)=0, \label{c1}\\
&& {}^{c}D^{\alpha}_{a+}h(x)|_{x=a}={}^{c}D^{\alpha}_{a+}h(x)|_{x=b}=0,\label{c2}
\end{eqnarray}
then $ \gamma(x)=c_{0}+c_{1}x$,
where $c_{0}, c_{1}$ are some real constants.
\end{lemma}

\begin{proof} Let us define function $h$ as follows
\begin{equation} \label{defh1}
 h(x):=I^{1+\alpha}_{a+}\left(\gamma(x)-c_{0}-c_{1}x\right)
\end{equation}
with constants fixed by the conditions
\begin{eqnarray}
&&I^{2}_{a+}\left(\gamma(x)-c_{0}-c_{1}x\right)|_{x=b}=0\label{cn1}\\
&& I^{1}_{a+}\left(\gamma(x)-c_{0}-c_{1}x\right)|_{x=b}=0.\label{cn2}
\end{eqnarray}
Observe that function $h$ is continuous and fulfills the boundary conditions
$$ h(a)=0\hspace{2cm} I^{1-\alpha}_{a+}h(x)|_{x=b}=I^{2}_{a+}\left(\gamma(x)-c_{0}-c_{1}x\right)|_{x=b}=0$$
and
$$ {}^{c}D^{\alpha}_{a+}h(x)|_{x=a}=D^{\alpha}_{a+}h(x)|_{x=a}=DI^{2}_{a+}\left(\gamma(x)-c_{0}-c_{1}x\right)|_{x=a}=$$
$$ =I^{1}_{a+}\left(\gamma(x)-c_{0}-c_{1}x\right)|_{x=a}=0$$
$$ {}^{c}D^{\alpha}_{a+}h(x)|_{x=b}=D^{\alpha}_{a+}h(x)|_{x=b}=DI^{2}_{a+}\left(\gamma(x)-c_{0}-c_{1}x\right)|_{x=b}=$$
$$ =  I^{1}_{a+}\left(\gamma(x)-c_{0}-c_{1}x\right)|_{x=b}=0.$$
In addition
$$ h'(x) = I^{\alpha}_{a+}(\gamma(x)-c_{0}-c_{1}x) \in C[a,b]$$
$$ D{}^{c}D^{\alpha}_{a+}h=\gamma(x)-c_{0}-c_{1}x\in C[a,b].$$
We also have
$$ \int_{a}^{b}\left (\gamma(x)-c_{0}-c_{1}x\right)D{}^{c}D^{\alpha}_{a+}h(x)dx= $$
$$ \int_{a}^{b}\left (-c_{0}-c_{1}x\right)D{}^{c}D^{\alpha}_{a+}h(x)dx=$$
$$ -c_{0}\cdot{}^{c}D^{\alpha}_{a+}h(x)|_{x=a}^{x=b}-c_{1}x\cdot{}^{c}D^{\alpha}_{a+}h(x)|_{x=a}^{x=b}+c_{1}\cdot I^{1-\alpha}_{a+}h(x)|_{x=a}^{x=b}=0.$$
On the other hand
$$ D{}^{c}D^{\alpha}_{a+}h(x)=D{}^{c}D^{\alpha}_{a+}I^{1+\alpha}_{a+}\left(\gamma(x)-c_{0}-c_{1}x\right)=\gamma(x)-c_{0}-c_{1}x$$
and
$$ 0= \int_{a}^{b}\left (\gamma(x)-c_{0}-c_{1}x\right)D{}^{c}D^{\alpha}_{a+}h(x)dx= \int_{a}^{b}\left (\gamma(x)-c_{0}-c_{1}x\right)^{2}dx.$$
Thus function $\gamma$ is
$$ \gamma(x)=c_{0}+c_{1}x.$$
\end{proof}

\begin{lemma}\label{lem:B}
Let $\alpha\in \left(\frac{1}{2},1\right) $, $\gamma\in C[a,b]$ and $D^{1-\alpha}_{a+}\gamma\in L^{2}[a,b]$. If
\begin{equation*}
\int_{a}^{b}\gamma(x)D{}^{c}D^{\alpha}_{a+}h(x)dx=0
\end{equation*}
for each $ h\in  C^{1}[a,b]$ such that $h''\in L^{2}[a,b]$ and $D{}^{c}D^{\alpha}_{a,+}h\in C[a,b]$
fulfilling boundary conditions \eqref{c1}, \eqref{c2}, then $ \gamma(x)=c_{0}+c_{1}x$,
where $c_{0}, c_{1}$ are some real constants.
\end{lemma}

\begin{proof}
We define function $h$ as in the proof of Lemma~\ref{lem:A}
\begin{equation}\label{defh2}
 h(x):=I^{1+\alpha}_{a+}\left(\gamma(x)-c_{0}-c_{1}x\right)
\end{equation}
with constants fixed by the conditions (\ref{cn1}) and (\ref{cn2})
The proof of the lemma is analogous to that of  Lemma~\ref{lem:A}.
In addition
for the second order derivative we have
$$ h''(x)=DI^{\alpha}_{a+}(\gamma(x)-c_{0}-c_{1}x)=$$
$$ = D^{1-\alpha}_{a+}\left (\gamma(x)-c_{0}-c_{1}x\right)=$$
$$ = D^{1-\alpha}_{a+}\gamma(x)-(c_{0}+c_{1}a)\frac{(x-a)^{\alpha-1}}{\Gamma(\alpha)}-c_{1}\frac{(x-a)^{\alpha}}{\Gamma(\alpha+1)}.$$
Let us observe that for $\alpha>1/2$
$$ \frac{(x-a)^{\alpha-1}}{\Gamma(\alpha)}\in L^{2}[a,b]$$
$$ \frac{(x-a)^{\alpha}}{\Gamma(\alpha+1)}\in C[a,b]\subset L^{2}[a,b].$$
Thus, we conclude that $h''\in L^{2}[a,b]$ and function $h$ constructed in this proof fulfills
all the assumptions of Lemma~\ref{lem:B}. The remaining part of proof is analogous to that for Lemma~\ref{lem:A}.
\end{proof}

\begin{lemma}\label{lem:C}
(1)
Let $\alpha\in \left (\frac{1}{2},1\right)$,  functions $\gamma_{j}\in C[a,b], \; j=1,2,3$ and $D^{1-\alpha}_{a+}\gamma_{3}\in L^{2}[a,b]$. If
\begin{equation}\label{eq:lem:C}
\int_{a}^{b}\left (\gamma_{1}(x)h(x)+\gamma_{2}(x){}^{c}D^{\alpha}_{a+}h(x)+\gamma_{3}(x)D{}^{c}D^{\alpha}_{a+}h(x)\right)dx=0
\end{equation}
for each $ h\in  C^{1}[a,b]$, such that  $h''\in L^{2}[a,b]$ and $D{}^{c}D^{\alpha}_{a,+}h\in C[a,b]$,
fulfilling boundary conditions \eqref{c1}, \eqref{c2},
then $\gamma_{3}\in C^{1}[a,b]$.\\
(2) Let $\alpha\in \left (\frac{1}{2},1\right)$ and functions $\gamma_{1,2}\in C[a,b]$. If
\begin{equation}\label{eq:lem:C2}
\int_{a}^{b}\left (\gamma_{1}(x)h(x)+\gamma_{2}(x){}^{c}D^{\alpha}_{a+}h(x)\right)dx=0
\end{equation}
for each $ h\in  C^{1}[a,b]$, such that $h''\in L^{2}[a,b]$ and $D{}^{c}D^{\alpha}_{a,+}h\in C[a,b]$,
fulfilling boundary conditions \eqref{c1}, \eqref{c2}
then
\begin{equation*}
\gamma_{1}(x)+{}^{c}D^{\alpha}_{b-}\gamma_{2}(x)=0.
\end{equation*}

\end{lemma}

\begin{proof}
Observe that integral \eqref{eq:lem:C} can be rewritten as follows
$$\int_{a}^{b}\left (\gamma_{1}(x)h(x)+\gamma_{2}(x){}^{c}D^{\alpha}_{a+}h(x)+\gamma_{3}(x)D{}^{c}D^{\alpha}_{a+}h(x)\right)dx= $$
$$ = \int_{a}^{b}\left (-I^{1}_{a+}I^{\alpha}_{b-}\gamma_{1}(x)-I^{1}_{a+}\gamma_{2}(x)+\gamma_{3}(x)\right )D{}^{c}D^{\alpha}_{a+}h(x)dx=0.$$
due to the fact that relations
$$I^{\alpha}_{a+}I^{1}_{b-} D{}^{c}D^{\alpha}_{a+}h(x)=-h(x)$$
and
$$I^{1}_{b-}D {}^{c}D^{\alpha}_{a+}h(x)=-{}^{c}D^{\alpha}_{a+}h(x)$$
are valid because function $h$ fulfills boundary conditions \eqref{c1},\eqref{c2}. Denote
$$ \gamma(x):= -I^{1}_{a+}I^{\alpha}_{b-}\gamma_{1}(x)-I^{1}_{a+}\gamma_{2}(x)+\gamma_{3}(x).$$
It is clear that $\gamma\in C[a,b]$ and $D^{1-\alpha}_{a+}\gamma \in L^{2}[a,b]$. Thus,
according to Lemma~\ref{lem:B}, there exist constants $c_{0}$ and $c_{1}$  such that
 $$ -I^{1}_{a+}I^{\alpha}_{b-}\gamma_{1}(x)-I^{1}_{a+}\gamma_{2}(x)+\gamma_{3}(x)=c_{0}+c_{1}x.$$
 Let us note that function $\gamma_{3}$ is
 $$ \gamma_{3}(x)=I^{1}_{a+}I^{\alpha}_{b-}\gamma_{1}(x)+I^{1}_{a+}\gamma_{2}(x)+c_{0}+c_{1}x.$$
Hence its first order derivative is continuous in $[a,b]$ and $\gamma_{3}\in C^{1}[a,b]$.\\
The proof of part (2) is similar. We write integral \eqref{eq:lem:C2}
as follows:
$$\int_{a}^{b}\left (\gamma_{1}(x)h(x)+\gamma_{2}(x){}^{c}D^{\alpha}_{a+}h(x)\right)dx
 = \int_{a}^{b}\left (-I^{1}_{a+}I^{\alpha}_{b-}\gamma_{1}(x)-I^{1}_{a+}\gamma_{2}(x)\right )D {}^{c}D^{\alpha}_{a+}h(x)dx=0. $$
The function in brackets is continuous in $[a,b]$  and
$$D^{1-\alpha}_{a+} \left (-I^{1}_{a+}I^{\alpha}_{b-}\gamma_{1}-I^{1}_{a+}\gamma_{2}\right )=-I^{\alpha}_{a+}I^{\alpha}_{b-}\gamma_{1}-I^{\alpha}_{a+}\gamma_{2}\in C[a,b] \subset L^{2}[a,b]$$
so we again can apply Lemma~\ref{lem:B} and obtain that there exist constants $c_{0}$
and $c_{1}$ such that
$$ I^{1}_{a+}I^{\alpha}_{b-}\gamma_{1}(x)+I^{1}_{a+}\gamma_{2}(x)=c_{0}+c_{1}x.$$
Thus functions $\gamma_{1,2}$
  fulfill equation:
$$ {}^{c}D^{\alpha}_{b-}\gamma_{2}(x)+\gamma_{1}(x)=0.$$
\end{proof}

The crucial idea in the proof of our main result is to apply direct variational methods to the fractional Sturm--Liouville equation. Starting from the fractional Sturm--Liouville equation the approach is to find an associated functional and to use this to find approximations to the stationary functions, which are necessarily solutions to the original equation. In the case of the fractional Sturm--Liouville equation an associated variational problem is the fractional isoperimetric problem which is defined in the following way:\\
\begin{equation}\label{izo:1}
\min J(y)=\int_a^b F(x,y(x),{}^{c}D^{\alpha}_{a+}y(x))\;dx,
\end{equation}
subject to the boundary conditions
\begin{equation} \label{izo:2}
y(a)=0,\quad y(b)=0
\end{equation}
and the isoperimetric constraint
\begin{equation}\label{izo:3}
I(y)=\int_a^b G(x,y(x),{}^{c}D^{\alpha}_{a+}y(x))\;dx=\xi,
\end{equation}
where $\xi\in\R$ is given, and
\begin{equation*}
\fonction{F}{[a,b]\times\R^2}{\R}{(x,y,u)}{F(x,y,u),}
\end{equation*}
\begin{equation*}
\fonction{G}{[a,b]\times\R^2}{\R}{(x,y,u)}{G(x,y,u)}
\end{equation*}
are functions of $C^1$ class, such that $\frac{\partial F}{\partial u}, \frac{\partial G}{\partial u}$ have continuous $D^{\alpha}_{b-}$ derivatives.

\begin{theorem}[cf. Theorem~3.3 \cite{Isoperimetric}]
\label{thm:EL}
If $y\in C[a,b]$ with $ {}^{c}D^{\alpha}_{a+}y\in C[a,b]$ is a stationary function for problem \eqref{izo:1}--\eqref{izo:3}, then there exists a real constant $\lambda$ such that, for $H=F+\lambda G$, the following equation:
\begin{equation}\label{eq:EL}
\frac{\partial H}{\partial y}(x,y(x),{}^{c}D^{\alpha}_{a+}y(x))+D^{\alpha}_{b-}\left[\frac{\partial H}{\partial u}(x,y(x),{}^{c}D^{\alpha}_{a+}y(x))\right]=0,
\end{equation}
holds, provided that
\begin{equation*}
\frac{\partial G}{\partial y}(x,y(x),{}^{c}D^{\alpha}_{a+}y(x))+D^{\alpha}_{b-}\left[\frac{\partial G}{\partial u}(x,y(x),{}^{c}D^{\alpha}_{a+}y(x))\right]\neq0.
\end{equation*}
\end{theorem}

%-------------------------------------------------

\section{Main results}\label{sec:mr}

For the first works on the fractional calculus of variations, we should look back to 1996, when Riewe used non--integer order derivatives in order to better describe nonconservative systems in mechanics \cite{CD:Riewe:1996}. Since then, many papers have been written on the subject,
%-- see, e.g.,~\cite{Shakoor:01,Isoperimetric,Blaszczyk:et:al,BCGI,LTDExist,Cresson,Gastao,Malgorzata1,Lazo,Agnieszka,Dorota,Tatiana,Rabei} and references therein.
for a state of the art on the fractional calculus of variations  we refer the reader to the recent book \cite{book:AD}. In this section, we will present an interesting application of the fractional variational calculus. Namely, using the fact that the fractional Sturm--Liouville eigenvalue problem can be remodeled as the fractional isoperimetric variational problem, we will show that there exist an increasing sequence of eigenvalues and a corresponding sequence of eigenfunctions for which the fractional Sturm--Liouville equation is satisfied. Moreover, we will prove that the lowest eigenvalue is the minimum value for a certain variational functional.

\subsection{Existence of Discrete Spectrum for Fractional Sturm--Liouville Problem}

We will show that similar to the classical case for the fractional Sturm--Liouville problem  there exist an infinite monotonic increasing sequence of eigenvalues. Moreover, that apart from multiplicative factors to each eigenvalue there corresponds precisely one eigenfunction.\\
We shall use the following assumptions.\\
(H1) Let $\frac{1}{2}<\alpha<1$ and $p,q,w_{\alpha}$ be given functions such that: $p$ is of $C^1$ class and $p(x)>0$; $q,w_{\alpha}$ are continuous, $w_{\alpha}(x)>0$ and $(\sqrt{w_{\alpha}})'$ is H\"{o}lderian of order $\beta\leq \alpha - \frac{1}{2}$.\\
Consider the fractional differential equation
\begin{equation}\label{eq:SLE}
\left[{}^{c}D^{\alpha}_{b-}p(x){}^{c}D^{\alpha}_{a+}+q(x)\right]y(x)= \lambda w_{\alpha}(x)y(x),
\end{equation}
that will be called the fractional Sturm--Liouville equation, subject to the boundary conditions
\begin{equation}\label{eq:BC}
y(a)=y(b)=0.
\end{equation}

\begin{theorem}\label{thm:exist}
Under assumptions (H1), the fractional Sturm--Liouville Problem (FSLP) \eqref{eq:SLE}--\eqref{eq:BC}
has an infinite increasing sequence of eigenvalues $\lambda^{(1)}, \lambda^{(2)},...,$
and to each eigenvalue $\lambda^{(n)}$ there corresponds an eigenfunction
$y^{(n)}$ which is unique up to a constant factor. Furthermore, eigenfunctions $y^{(n)}$ form an orthogonal set of solutions.
\end{theorem}

\begin{proof}
The proof is similar in spirit to \cite{book:GF} and will be divided into 6 steps. As in \cite{book:GF} at the same time we shall derive a method for approximating the eigenvalues and eigenfunctions.\\
\textit{Step 1.} We shall consider problem of minimizing the functional
\begin{equation}\label{eq:IF}
J(y)= \int_{a}^{b}\left[p(x) ({}^{c}D^{\alpha}_{a+}y)^{2}+q(x)y^{2}\right]dx
\end{equation}
subject to an isoperimetric constraint
\begin{equation}\label{eq:IC}
\hspace{2cm}I(y)= \int_{a}^{b}w_{\alpha}(x)y^{2}dx=1
\end{equation}
and boundary conditions \eqref{eq:BC}. First, let us point out that functional \eqref{eq:IF} is bounded from below. Indeed, as $p(x)>0$ we have
\begin{multline*}
J(y)= \int_{a}^{b}\left[p(x) ({}^{c}D^{\alpha}_{a+}y)^{2}+q(x)y^{2}\right]\;dx\\
\geq\min\limits_{x\in[a,b]}\frac{q(x)}{w_{\alpha}(x)}\cdot\int_{a}^{b}w_{\alpha}(x)y^{2}\;dx
=\min\limits_{x\in[a,b]}\frac{q(x)}{w_{\alpha}(x)}=:M_0>-\infty.
\end{multline*}
From now on, for simplicity, we assume that $a=0$ and $b=\pi$. According to the Ritz method, we approximate solution of \eqref{eq:BC}--\eqref{eq:IC}  using the following trigonometric function with coefficient
depending on $w_{\alpha}$:
\begin{equation}\label{aprox}
y_{m}(x)= \frac{1}{\sqrt{w_{\alpha}}}\sum_{k=1}^{m} \beta_{k}\sin(kx).
\end{equation}
Observe that $y_{m}(0)=y_{m}(\pi)=0$. Substituting \eqref{aprox} into \eqref{eq:IF} and \eqref{eq:IC} we obtain the problem of minimizing the function
\begin{equation}\label{eq:F}
\tilde{J}(\beta_{1},...,\beta_{m})= \tilde{J}([\beta])=
\end{equation}
$$ = \sum_{k,j=1}^{m}\beta_{k}\beta_{j}\int_{0}^{\pi}\left[p(x) \left({}^{c}D^{\alpha}_{0+}\frac{\sin (kx)}{\sqrt{w_{\alpha}}}\cdot {}^{c}D^{\alpha}_{0+}\frac{\sin (jx)}{\sqrt{w_{\alpha}}}\right)+
\frac{q(x)}{w_{\alpha}(x)}\sin(kx)\sin(jx)\right]dx$$
subject to the condition
\begin{equation}\label{eq:sphere}
\tilde{I}(\beta_{1},...,\beta_{m})= \tilde{I}([\beta])=\frac{\pi}{2}\sum_{k=1}^{m}(\beta_{k})^{2}=1.
\end{equation}
Since $\tilde{J}([\beta])$ is continuous and the set given by \eqref{eq:sphere} is compact, function $\tilde{J}([\beta])$ attains minimum, denoted by $\lambda_m^{(1)}$, at some point $[\beta^{(1)}]=(\beta_{1}^{(1)},...,\beta_{m}^{(1)})$. If this procedure is carried out for $m=1,2,\ldots$, we obtain a sequence of numbers $\lambda_{1}^{(1)},\lambda_2^{(1)},\ldots$. Because $\lambda_{m+1}^{(1)}\leq \lambda_m^{(1)}$ and  $J(y)$ is bounded from below, we can find the limit
\begin{equation*}
\lim\limits_{m\rightarrow\infty}\lambda_m^{(1)}=\lambda^{(1)}.
\end{equation*}

\textit{Step 2.} Let
\begin{equation*}
y_{m}^{(1)}(x)= \frac{1}{\sqrt{w_{\alpha}}}\sum_{k=1}^{m} \beta_{k}^{(1)}\sin(kx)
\end{equation*}
denote the linear combination \eqref{aprox} achieving the minimum $\lambda_m^{(1)}$. We shall prove that sequence $(y_{m}^{(1)})_{m\in\N}$ contains a uniformly convergent subsequence. From now on, for simplicity, we will write $y_m$ instead of $y_m^{(1)}$. Recall that
\begin{equation*}
\lambda_m^{(1)}=\int_{0}^{\pi} \left[p(x)\left(\Dl\seq\right)^2+q(x)\seq^2\right]dx
\end{equation*}
is convergent, so it must be bounded, i.e.,~ there exists constant $M>0$ such that
\begin{equation*}
\int_{0}^{\pi} \left[p(x)\left(\Dl\seq\right)^2+q(x)\seq^2\right]dx\leq M,~m\in\N.
\end{equation*}
Therefore, for all $m\in\N$ it holds the following:
\begin{multline*}
\int_0^{\pi} p(x)\left(\Dl\seq\right)^2\;dx\leq M+\left|\int_0^{\pi} q(x)\seq^2dx\right|\\
\leq M+\max\limits_{x\in[0,\pi]}\left|\frac{q(x)}{w_{\alpha}(x)}\right|\int_0^{\pi} w_{\alpha}(x)\seq^2\;dx=M+\max\limits_{x\in[0,\pi]}\left|\frac{q(x)}{w_{\alpha}(x)}\right|=:M_1.
\end{multline*}
Moreover, since $p(x)>0$ we have
\begin{equation*}
\min\limits_{x\in[0,\pi]}p(x)\int_0^{\pi} \left(\Dl\seq\right)^2\;dx\leq \int_0^{\pi} p(x)\left(\Dl\seq\right)^2\;dx\leq M_1,
\end{equation*}
and hence
\begin{equation}\label{eq:4}
\int_0^{\pi} \left(\Dl\seq\right)^2\;dx\leq\frac{M_1}{\min\limits_{x\in[0,\pi]}p(x)}=: M_2.
\end{equation}
Using \eqref{eq:3}, \eqref{eq:4}, condition $\seq(0)=0$ and Schwartz inequality, we obtain the following:
\begin{multline*}
\left|\seq(x)\right|^2=\left|\Il \Dl\seq(x)\right|^2=\frac{1}{\Gamma(\alpha)}\left|\int_0^{x}(x-t)^{\alpha-1}\;\Dl\seq(t)dt\right|^2\\
\leq \frac{1}{\Gamma(\alpha)}\left(\int_0^\pi \left|\Dl\seq(t)\right|^2dt\right)\left(\int_0^x (x-t)^{2(\alpha-1)}dt\right)\\
\leq \frac{1}{\Gamma(\alpha)}M_2\int_0^x (x-t)^{2(\alpha-1)}dt<\frac{1}{\Gamma(\alpha)}M_2\frac{1}{2\alpha-1}\pi^{2\alpha-1},
\end{multline*}
so that $(\seq)_{m\in\N}$ is uniformly bounded.
Now, using Schwartz inequality, equation \eqref{eq:4} and the fact that the following inequality:
\begin{equation*}%\label{eq:5}
\forall t_1\geq t_2\geq 0,~(t_1-t_2)^2\leq t_1^2-t_2^2
\end{equation*}
holds, we have for any $0< x_1< x_2\leq \pi $ that
\begin{multline*}
\left|\seq(x_2)-\seq(x_1)\right|=\left|\Il \Dl\seq(x_2)-\Il \Dl\seq(x_1)\right|\\
=\frac{1}{\Gamma(\alpha)}\left|\int_0^{x_2}(x_2-t)^{\alpha-1}\Dl\seq(t)dt-\int_0^{x_1}(x_1-t)^{\alpha-1}\Dl\seq(t)dt\right|\\
=\frac{1}{\Gamma(\alpha)}\left|\int_{x_1}^{x_2}(x_2-t)^{\alpha-1}\Dl\seq(t)dt-\int_0^{x_1}\left((x_2-t)^{\alpha-1}-(x_1-t)^{\alpha-1}\right)\Dl\seq(t)dt\right|\\
\leq \frac{1}{\Gamma(\alpha)}\left[\left(\int_{x_1}^{x_2}(x_2-t)^{2(\alpha-1)}dt\right)^{\frac{1}{2}}\left(\int_{x_1}^{x_2}\left[\left(\Dl\seq(t)\right)^2\right]dt\right)^{\frac{1}{2}}\right.\\
\left.+\left(\int_0^{x_1}\left((x_1-t)^{\alpha-1}-(x_2-t)^{\alpha-1}\right)^2dt\right)^{\frac{1}{2}}\left(\int_{0}^{x_1}\left[\left(\Dl\seq(t)\right)^2\right]dt\right)^{\frac{1}{2}}\right]\\
\leq\frac{\sqrt{M_2}}{\Gamma(\alpha)}\left[\left(\int_{x_1}^{x_2}(x_2-t)^{2(\alpha-1)}dt\right)^{\frac{1}{2}}+\left(\int_0^{x_1}\left((x_1-t)^{2(\alpha-1)}-(x_2-t)^{2(\alpha-1)}\right)dt\right)^{\frac{1}{2}}\right]\\
=\frac{\sqrt{M_2}}{\Gamma(\alpha)\sqrt{2\alpha-1}}\left[(x_2-x_1)^{\alpha-\frac{1}{2}}+\left[(x_2-x_1)^{2\alpha-1}-x_2^{2\alpha-1}+x_1^{2\alpha-1}\right]^{\frac{1}{2}}\right]\\
\leq \frac{2\sqrt{M_2}}{\Gamma(\alpha)\sqrt{2\alpha-1}}(x_2-x_1)^{\alpha-\frac{1}{2}}.
\end{multline*}
Therefore, by Ascoli's theorem, there exists a uniformly convergent subsequence $(y_{m_n})_{n\in\N}$ of sequence $(\seq)_{m\in\N}$. It means that we can find $y^{(1)}\in C[0,\pi]$ such that
\begin{equation*}
y^{(1)}=\lim\limits_{n\rightarrow\infty}y_{m_n}.
\end{equation*}

\textit{Step 3.} Observe that by the Lagrange multiplier rule at $[\beta]=[\beta^{(1)}]$ we have
\begin{equation*}
0=\frac{\partial}{\partial \beta_{j}}\left [\tilde{J}([\beta])-\lambda^{(1)}_{m}\tilde{I}([\beta])\right]|_{[\beta]=[\beta^{(1)}]},~~j=1,\dots,m.
\end{equation*}
Multiplying each of equations by an arbitrary constant $C^{j}$ and summing from 1 to $m$ we obtain
\begin{equation}\label{lag:1}
0=\sum_{j=1}^{m}C^{j}\frac{\partial}{\partial \beta_{j}}\left [\tilde{J}([\beta])-\lambda^{(1)}_{m}\tilde{I}([\beta])\right]|_{[\beta]=[\beta^{(1)}]}.
\end{equation}
Introducing
\begin{equation*}
h_{m}(x)=\frac{1}{\sqrt{w_{\alpha}}}\sum_{j=1}^{m} C^{j}\sin (jx)
\end{equation*}
we can rewrite \eqref{lag:1} in the form
\begin{equation}\label{lag:2}
0=\int_{0}^{\pi}\left [p(x){}^{c}D^{\alpha}_{0+}y_{m}{}^{c}D^{\alpha}_{0+}h_{m}+[q(x)-\lambda^{(1)}_{m}w_{\alpha}(x)]y_{m}h_{m}\right]dx.
\end{equation}
Using the differentiation properties and formula  ${}^{c}D^{\alpha}_{0+}y_{m}=D^{\alpha}_{0+}y_{m}=DI^{1-\alpha}_{0+}y_{m}$ we write \eqref{lag:2} as
\begin{equation}\label{intseq}
0=\int_{0}^{\pi}\left [-p'(x)I^{1-\alpha}_{0+}y_{m}{}^{c}D^{\alpha}_{0+}h_{m}
-p(x)I^{1-\alpha}_{0+}y_{m}D{}^{c}D^{\alpha}_{0+}h_{m}\right]dx +
\end{equation}
$$ + p(x)I^{1-\alpha}_{0+}y_{m}{}^{c}D^{\alpha}_{0+}h_{m}|_{x=0}^{x=\pi}+\int_{0}^{\pi}[q(x)-\lambda^{(1)}_{m}w_{\alpha}(x)]y_{m}h_{m}dx:=I_{m}.$$
By Lemma~\ref{lem:D} (with $w=1/\sqrt{w_{\alpha}}$) and Lemma~\ref{lem:E} (Appendix), for function $h$ fulfilling assumptions of
Lemma~\ref{lem:B}, we shall obtain in the limit (at least for the convergent subsequence $(y_{m_n})_{n\in\N}$) the relation
\begin{equation}
0=\int_{0}^{\pi}\left [-p'(x)I^{1-\alpha}_{0+}y^{(1)}\;{}^{c}D^{\alpha}_{0+}h
-p(x)I^{1-\alpha}_{0+}y^{(1)}D{}^{c}D^{\alpha}_{0+}h\right]dx +
\label{int}
\end{equation}
$$ + \int_{0}^{\pi}[q(x)-\lambda^{(1)}w_{\alpha}(x)]y^{(1)}hdx:=I.$$
Let us check the convergence of integrals \eqref{intseq} explicitly
\begin{equation}\label{estproof}
 \left |I_{m}-I\right|
\end{equation}
$$ \leq \int_{0}^{\pi}|-p'(x)I^{1-\alpha}_{0+}y_{m}{}^{c}D^{\alpha}_{0+}h_{m}+p'(x)I^{1-\alpha}_{0+}y^{(1)}\;{}^{c}D^{\alpha}_{0+}h|dx+$$
$$ + \int_{0}^{\pi}|p(x)I^{1-\alpha}_{0+}y_{m}D{}^{c}D^{\alpha}_{0+}h_{m}-p(x)I^{1-\alpha}_{0+}y^{(1)}\;D{}^{c}D^{\alpha}_{0+}h|dx+$$
$$ +\left|p(x)I^{1-\alpha}_{0+}y_{m}{}^{c}D^{\alpha}_{0+}h_{m}|_{x=0}-p(x)I^{1-\alpha}_{0+}y^{(1)}\;{}^{c}D^{\alpha}_{0+}h|_{x=0}\right|+$$
$$+\left|p(x)I^{1-\alpha}_{0+}y_{m}{}^{c}D^{\alpha}_{0+}h_{m}|_{x=\pi}-p(x)I^{1-\alpha}_{0+}y^{(1)}\;{}^{c}D^{\alpha}_{0+}h|_{x=\pi}\right|+$$
$$ + \int_{0}^{\pi}|[q(x)-\lambda^{(1)}_{m}w_{\alpha}(x)]y_{m}h_{m}-[q(x)-\lambda^{(1)}w_{\alpha}(x)]y^{(1)}\; h|dx.$$

For the first integral we get
$$\int_{0}^{\pi}|-p'(x)I^{1-\alpha}_{0+}y_{m}{}^{c}D^{\alpha}_{0+}h_{m}+p'(x)I^{1-\alpha}_{0+}y^{(1)}\;{}^{c}D^{\alpha}_{0+}h|dx\leq $$
$$ \leq ||p'||\cdot\left [ ||{}^{c}D^{\alpha}_{0,+}h||\cdot ||I^{1-\alpha}_{0,+}(y_{m}-y^{(1)})||_{L^{1}}+M_3K_{1-\alpha}\sqrt{\pi}||{}^{c}D^{\alpha}_{0+}(h_{m}-h)||_{L^{2}}\right],$$
where constant $M_3=\sup\limits_{m\in \mathbb{N}} ||y_{m}||$ and $||\cdot ||$ denotes the supremum norm in the $C[0,\pi]$ space.
Now, we estimate the second integral
$$\int_{0}^{\pi}|p(x)I^{1-\alpha}_{0+}y_{m}D{}^{c}D^{\alpha}_{0+}h_{m}-p(x)I^{1-\alpha}_{0+}y^{(1)}\;D{}^{c}D^{\alpha}_{0+}h|dx\leq $$
$$ \leq ||p||\cdot\left [||D{}^{c}D^{\alpha}_{0+}h||_{L^{2}}\cdot ||I^{1-\alpha}_{0+}(y_{m}-y^{(1)})||_{L^{2}}+M_3K_{1-\alpha}\cdot ||D{}^{c}D^{\alpha}_{0+}(h_{m}-h)||_{L^{1}}\right].$$
For the next two terms we have
\begin{equation}\label{point1}
 I^{1-\alpha}_{0,+}y_{m}(0)\longrightarrow I^{1-\alpha}_{0,+}y(0),\quad  I^{1-\alpha}_{0,+}y_{m}(\pi)\longrightarrow I^{1-\alpha}_{0,+}y(\pi)
\end{equation}
resulting from the convergence of sequence $ ||y_{m}-y||\longrightarrow 0$. For sequence $h_{m}=g_{m}/\sqrt{w_{\alpha}}$, we infer from
Lemma~\ref{lem:E} that
 $$ \lim_{m\longrightarrow\infty}||h'_{m}-h'||=0.$$
Hence, also
 $$ \lim_{m\longrightarrow\infty}||{}^{c}D^{\alpha}_{0,+}(h_{m}-h)||= \lim_{m\longrightarrow\infty}||I^{1-\alpha}_{0,+}(h'_{m}-h')||=0$$
and at points $x=0,\pi$ we obtain
\begin{equation}\label{point2}
  {}^{c}D^{\alpha}_{0,+}h_{m}(0)\longrightarrow {}^{c}D^{\alpha}_{0,+}h(0),\quad  {}^{c}D^{\alpha}_{0,+}h_{m}(\pi)\longrightarrow {}^{c}D^{\alpha}_{0,+}h(\pi).
\end{equation}
The above pointwise convergences \eqref{point1} and \eqref{point2} imply that
$$ \lim_{m\longrightarrow\infty} \left|p(x)I^{1-\alpha}_{0+}y_{m}{}^{c}D^{\alpha}_{0+}h_{m}|_{x=0}-p(x)I^{1-\alpha}_{0+}y^{(1)}\;{}^{c}D^{\alpha}_{0+}h|_{x=0}\right|=0$$
$$  \lim_{m\longrightarrow\infty}\left|p(x)I^{1-\alpha}_{0+}y_{m}{}^{c}D^{\alpha}_{0+}h_{m}|_{x=\pi}-p(x)I^{1-\alpha}_{0+}y^{(1)}\;{}^{c}D^{\alpha}_{0+}h|_{x=\pi}\right|=0. $$
Finally, for the last term in estimation \eqref{estproof} we get
$$ \int_{0}^{\pi}|[q(x)-\lambda^{(1)}_{m}w_{\alpha}(x)]y_{m}h_{m}-[q(x)-\lambda^{(1)}w_{\alpha}(x)]y^{(1)}\; h|dx\leq$$
$$ \leq \int_{0}^{\pi}|q(x)(y_{m}h_{m}-y^{(1)}\; h)|dx
 +\int_{0}^{\pi}|w_{\alpha}(x)(\lambda^{(1)}_{m}y_{m}h_{m}-\lambda^{(1)}y^{(1)}\; h)|dx \leq $$
$$ \leq \pi\cdot||q||\cdot \left [M_3 \cdot ||h_{m}-h|| + ||h||\cdot ||y_{m}-y^{(1)}||\right] +$$
$$ +\pi\cdot||w_{\alpha}||\cdot\left  [\Lambda\left (M_3 \cdot ||h_{m}-h|| + ||h||\cdot ||y_{m}-y^{(1)}||\right) + ||y^{(1)}h||\cdot |\lambda^{(1)}_{m}-\lambda^{(1)}|\right],$$
where constants $M_3=\sup\limits_{m\in \mathbb{N}} ||y_{m}||$ and $\Lambda=\sup\limits_{m\in \mathbb{N}} |\lambda^{(1)}_{m}| $. We conclude
that
$$ 0= \lim_{m\longrightarrow \infty}I_{m}=I$$
and (\ref{int}) is fulfilled for function $y^{(1)}$ being the limit of subsequence $(y_{m_n})$ of the sequence $\left(y_{m}\right)_{m\in\N}$.

\textit{Step 4.}
Let us denote in relation \eqref{int}:
\begin{eqnarray*}
&& \gamma_1 (x):=(q(x)-\lambda^{(1)}w_{\alpha})y^{(1)}(x),\\
&& \gamma_2 (x):= -p'(x)I^{1-\alpha}_{0+} y^{(1)}(x),\\
&& \gamma_3 (x):= -p(x)I^{1-\alpha}_{0+}y^{(1)}(x).
\end{eqnarray*}
We observe that $\gamma_{j} \in C[0,\pi],
\;j=1,2,3 $ and $D^{1-\alpha}_{0+}\gamma_{3}\in L^{2}[0,\pi]$ because
$$ D^{1-\alpha}_{0+}\gamma_{3}=D^{1-\alpha}_{0+}\left (p \cdot I^{1-\alpha}_{0+}y^{(1)}\right) = I^{\alpha}_{0+} D\left (p \cdot I^{1-\alpha}_{0+}y^{(1)}\right)=$$
$$  = I^{\alpha}_{0+} \left (p' \cdot I^{1-\alpha}_{0+}y^{(1)} + p\cdot
\; {}^{c}D^{\alpha}_{0+}y^{(1)}\right). $$
Both parts of the above function belong to the $L^{2}[0,\pi]$ space.\\
Assuming that function $h$ in \eqref{int} is an arbitrary function fulfilling assumptions of Lemma~\ref{lem:C} and  applying Lemma~\ref{lem:C} part (1),  we conclude that
$\gamma_3 =-p\cdot I^{1-\alpha}_{0+}y^{(1)}\in C^{1}[0,\pi]$ . From this fact it follows that $p\cdot D I^{1-\alpha}_{0+}y^{(1)}\in C[0,\pi]$ and
  integral \eqref{int} can be rewritten as
\begin{equation*}
0=\int_{0}^{\pi}\left [p(x\; {}^{c}D^{\alpha}_{0+}y^{(1)}{}^{c}D^{\alpha}_{0+}h
+(q(x)-\lambda^{(1)}w_{\alpha}(x))y^{(1)}h\right]dx.
\end{equation*}
Now we apply Lemma~\ref{lem:C} part (2) defining
\begin{eqnarray*}
&&\bar{\gamma}_{1} (x) : =  \gamma_1 (x)=(q(x)-\lambda^{(1)}w_{\alpha})y^{(1)}(x),\\
&&\bar{\gamma}_{2} (x):= p(x)DI^{1-\alpha}_{0+}y^{(1)}(x).
\end{eqnarray*}
This time $\bar{\gamma}_{1,2}\in C[0,\pi]$ and from Lemma~\ref{lem:C} part (2)
it follows that
\begin{equation*}
\left[{}^{c}D^{\alpha}_{\pi-}p(x){}^{c}D^{\alpha}_{0+}+q(x)\right]y^{(1)}(x)= \lambda^{(1)}w_{\alpha}(x)y^{(1)}(x).
\end{equation*}
By construction this solution fulfills the Dirichlet boundary conditions
\begin{equation*}
y^{(1)}(0)=y^{(1)}(\pi)=0
\end{equation*}
and is nontrivial because
\begin{equation*}
I(y^{(1)})=\int_{0}^{\pi}w_{\alpha}(x)\left (y^{(1)}(x)\right)^{2}dx=1.
\end{equation*}
In addition, we also have for the solution
$$ D^{\alpha}_{0+}y^{(1)} ={}^{c}D^{\alpha}_{0+}y^{(1)}\in C[0,\pi]. $$
Let us observe that from the Dirichlet boundary conditions
it follows that $y^{(1)}$ also solves the FSLP (\ref{eq:SLE})-(\ref{eq:BC}) in $[0,\pi]$. \\
\textit{Step 5.} Now, let us restore the superscript on $y_{m}^{(1)}$ and show that $\left(y_{m}^{(1)}\right)_{m\in\N}$ itself converges to $y^{(1)}$. First, let us point out that for given $\lambda$ the solution of
 \begin{equation}\label{eq:SL:2}
\left[{}^{c}D^{\alpha}_{\pi-}p(x){}^{c}D^{\alpha}_{0+}+q(x)\right]y(x)= \lambda w_{\alpha}(x)y(x),
\end{equation}
subject to the boundary conditions
\begin{equation}\label{eq:BC:2}
y(0)=y(\pi)=0
\end{equation}
and the normalization condition
\begin{equation}\label{eq:IC:2}
\int_0^{\pi} w_{\alpha}(x)y^2\;dx=1
\end{equation}
is unique except for a sign. Next, let us assume that $y^{(1)}$ solves Sturm--Liouville equation \eqref{eq:SL:2} and that corresponding eigenvalue is $\lambda=\lambda^{(1)}$. In addition, suppose that $y^{(1)}$ is non trivial i.e., we can find $x_0\in [0,\pi]$ such that $y^{(1)}(x_0)\neq 0$ and choose the sign so that $y^{(1)}(x_0)>0$. Similarly, for all $m\in\N$, let $y_m^{(1)}$ solve \eqref{eq:SL:2} with corresponding eigenvalue $\lambda=\lambda_m^{(1)}$ and let us choose the signs so that $y_m^{(1)}(x_0)\geq 0$. Now, suppose that $\left(y_{m}^{(1)}\right)_{m\in\N}$ does not converge to $y^{(1)}$. It means that we can find another subsequence of $\left(y_{m}^{(1)}\right)_{m\in\N}$ such that it converges to another solution $\bar{y}^{(1)}$ of \eqref{eq:SL:2} with $\lambda=\lambda^{(1)}$. We know that for $\lambda=\lambda^{(1)}$ solution of \eqref{eq:SL:2} subject to \eqref{eq:BC:2} and \eqref{eq:IC:2} must be unique except for a sign, hence
\begin{equation*}
\bar{y}^{(1)}=-y^{(1)}
\end{equation*}
and we must have $\bar{y}^{(1)}(x_0)<0$. However, it is impossible because for all $m\in\N$ value of $y_m^{(1)}$ in $x_0$ is greater or equal zero. It means that we have contradiction and hence, choosing each $y_m^{(1)}$ with adequate sign, we obtain $y_m^{(1)}\rightarrow y^{(1)}$.

\textit{Step 6.} In order to find eigenfunction $y^{(2)}$ and the corresponding eigenvalue $\lambda^{(2)}$, we again minimize functional \eqref{eq:IF} subject to \eqref{eq:IC} and \eqref{eq:BC}, but now with an extra orthogonality condition
\begin{equation}\label{eq:OC}
\int_0^{\pi} w_{\alpha}(x)y(x) y^{(1)}(x)\;dx=0.
\end{equation}
If we approximate solution by
\begin{equation*}
y_{m}(x)= \frac{1}{\sqrt{w_{\alpha}}}\sum_{k=1}^{m} \beta_{k}\sin(kx),
\hspace{2cm} y_{m}(0)=y_{m}(\pi)=0,%~~m\in\N,
\end{equation*}
then we again receive quadratic form \eqref{eq:F}. However in this case admissible solutions are points satisfying \eqref{eq:sphere} together with
\begin{equation}\label{eq:hyperpl}
\frac{\pi}{2}\sum\limits_{k=1}^{m}\beta_k\beta_k^{(1)}=0,
\end{equation}
i.e., they lay in $(m-1)$-dimensional sphere. As before, we find that function $\tilde{J}([\beta])$ has a minimum $\lambda_m^{(2)}$ and there exists $\lambda^{(2)}$ such that
\begin{equation*}
\lambda^{(2)}=\lim\limits_{m\rightarrow\infty}\lambda_m^{(2)},
\end{equation*}
because $J(y)$ is bounded from below. Moreover, it is clear that the following relation:
\begin{equation}\label{eq:eig}
\lambda^{(1)}\leq\lambda^{(2)}
\end{equation}
holds. Now, let us denote by
\begin{equation*}
y_{m}^{(2)}(x)= \frac{1}{\sqrt{w_{\alpha}}}\sum_{k=1}^{m} \beta_{k}^{(2)}\sin(kx),
\end{equation*}
the linear combination achieving the minimum $\lambda_m^{(2)}$, where $\beta^{(2)}=(\beta_1^{(2)},\dots,\beta_m^{(2)})$ is the point satisfying \eqref{eq:sphere} and \eqref{eq:hyperpl}. By the same argument as before, we can prove that the sequence $(y_m^{(2)})_{m\in\N}$ converges uniformly to a limit function $y^{(2)}$, which satisfies the Strum-Liouville equation \eqref{eq:SLE} with $\lambda^{(2)}$, the boundary conditions \eqref{eq:BC}, normalization condition \eqref{eq:IC} and the orthogonality condition \eqref{eq:OC}. Therefore, solution $y^{(2)}$ of the FSLP corresponding to the eigenvalue $\lambda^{(2)}$ exists. Furthermore, because orthogonal functions cannot be linearly dependent, and since only one eigenfunction corresponds to each eigenvalue (except for a constant factor), we have the strict inequality
\begin{equation*}
\lambda^{(1)}<\lambda^{(2)}
\end{equation*}
instead of \eqref{eq:eig}. Finally, if we repeat the above procedure, with similar modifications, we can obtain eigenvalues $\lambda^{(3)},\lambda^{(4)},\dots$ and corresponding eigenfunctions $y^{(3)},y^{(4)},\dots$.
\end{proof}

%----------------------------------------------

\subsection{The First Eigenvalue}
\label{First Eigenvalue}

In this section we prove two theorems showing that the first eigenvalue of problem \eqref{eq:SLE}--\eqref{eq:BC} is a minimum value of certain functionals. As in the proof of Theorem~\ref{thm:exist} in the sequel, for simplicity, we assume that $a=0$ and $b=\pi$ in the problem \eqref{eq:SLE}--\eqref{eq:BC}.

\begin{theorem}\label{thm:FE}
Let $y^{(1)}$ be the eigenfunction, normalized to satisfy the isoperimetric constraint
\begin{equation}\label{eq:IC2}
I(y)=\int_0^{\pi}w_{\alpha}(x)y^2\;dx=1,
\end{equation}
associated to the first eigenvalue $\lambda^{(1)}$ of problem \eqref{eq:SLE}--\eqref{eq:BC} and assume that function $D^{\alpha}_{\pi-}(p(x){}^{c}D^{\alpha}_{0+}y)$ is continuous. Then, $y^{(1)}$ is a minimizer of the following variational functional:
\begin{equation}\label{eq:F2}
J(y)=\int_0^{\pi}\left[p(x) ({}^{c}D^{\alpha}_{0+}y)^{2}+q(x)y^{2}\right]\;dx ,
\end{equation}
in the class $C[0,\pi]$ with ${}^{c}D^{\alpha}_{0+}y\in C[0,\pi]$ subject to the boundary conditions
\begin{equation}\label{eq:BC2}
y(0)=y(\pi)=0
\end{equation}
and an isoperimetric constraint \eqref{eq:IC2}. Moreover,
\begin{equation*}
J(y^{(1)})=\lambda^{(1)}.
\end{equation*}
\end{theorem}

\begin{proof}
Suppose that $y\in C [0,\pi]$ is a minimizer of $J$ and ${}^{c}D^{\alpha}_{0+}y\in C[0,\pi]$. Then, by Theorem~\ref{thm:EL}, there is number $\lambda$ such that $y$ satisfies equation
\begin{equation}\label{eq:SLE0}
\left[D^{\alpha}_{\pi-}p(x){}^{c}D^{\alpha}_{0+}+q(x)\right]y(x)= \lambda w_{\alpha}(x)y(x),
\end{equation}
and conditions \eqref{eq:IC2}, \eqref{eq:BC2}. Since $D^{\alpha}_{\pi-}(p(x){}^{c}D^{\alpha}_{0+}y)$ and ${}^{c}D^{\alpha}_{\pi-}(p(x){}^{c}D^{\alpha}_{0+}y)$ are continuous, it follows that $p(x){}^{c}D^{\alpha}_{0+}y(x)\left|_{x=\pi}\right.=0$. Therefore equation \eqref{eq:SLE0} is equivalent to
\begin{equation}\label{eq:SLE2}
\left[{}^{c}D^{\alpha}_{\pi-}p(x){}^{c}D^{\alpha}_{0+}+q(x)\right]y(x)= \lambda w_{\alpha}(x)y(x).
\end{equation}
Let us multiply \eqref{eq:SLE0} by $y$ and integrate it on the interval $[0,\pi]$, then
\begin{equation*}
\int_0^{\pi}\left(y\cdot D^{\alpha}_{\pi-}(p(x){}^{c}D^{\alpha}_{0+}y)+q(x)y^2\right)\;dx=\lambda\int_0^{\pi}w_{\alpha}(x)y^2\;dx.
\end{equation*}
Applying the integration by the parts formula for fractional derivatives (cf. \eqref{eq:IBP}) and having in mind that conditions \eqref{eq:BC2}, \eqref{eq:IC2} and $p(x){}^{c}D^{\alpha}_{0+}y(x)\left|_{x=b}\right.=0$ hold, one has
\begin{equation*}
\int_0^{\pi}\left(\left(\Dl y\right)^2 p(x)+q(x)y^2\right)\;dx=\lambda.
\end{equation*}
Hence
\begin{equation*}
J(y)=\lambda.
\end{equation*}
Any solution to problem \eqref{eq:IC2}--\eqref{eq:BC2} which satisfies equation \eqref{eq:SLE2} must be nontrivial since \eqref{eq:IC2} holds, so $\lambda$ must be an eigenvalue. Moreover, according to Theorem~\ref{thm:exist} there is the least element in the spectrum being eigenvalue $\lambda^{(1)}$ and the corresponding eigenfunction $y^{(1)}$ normalized to meet the isoperimetric condition. Therefore $J(y^{(1)})=\lambda^{(1)}$.
\end{proof}

\begin{definition}
We will call functional $R$ defined by
\begin{equation*}
R(y)=\frac{J(y)}{I(y)},
\end{equation*}
where $J(y)$ is given by \eqref{eq:F2} and $I(y)$ by \eqref{eq:IC2}, the Rayleigh quotient for the fractional Sturm--Liouville problem \eqref{eq:SLE}--\eqref{eq:BC}.
\end{definition}

\begin{theorem}\label{thm:RQ}
Let us assume that function $y\in C[0,\pi]$ with ${}^{c}D^{\alpha}_{0+}y\in C[0,\pi]$, satisfying boundary conditions $y(0)=y(\pi)=0$ and being nontrivial, is a minimizer of Rayleigh quotient $R$ for the Sturm--Liouville problem \eqref{eq:SLE}--\eqref{eq:BC}. Moreover assume that function $D^{\alpha}_{\pi-}(p(x){}^{c}D^{\alpha}_{0+}y)$ is continuous. Then, value of $R$ in $y$ is equal to the first eigenvalue $\lambda^{(1)}$ i.e., $R(y)=\lambda^{(1)}$.
\end{theorem}

\begin{proof}
Suppose that function $y\in C [0,\pi]$ with ${}^{c}D^{\alpha}_{0+}y\in C[0,\pi]$, satisfying $y(0)=y(\pi)=0$ and nontrivial is a minimizer of Rayleigh quotient $R$ and that value of $R$ in $y$ is equal to $\lambda$, i.e.,
$$
R(y)=\frac{J(y)}{I(y)}=\lambda.
$$
Consider one-parameter family of curves
\begin{equation*}
\hat{y}=y+h\eta,~~\left|h\right|\leq\varepsilon,
\end{equation*}
where $\eta\in C^1[0,\pi]$ is such that $\eta(0)=\eta(\pi)=0$, $\eta\neq 0$ and define the following functions
\begin{equation*}
\fonction{\phi}{[-\varepsilon,\varepsilon]}{\R}{h}{I(y+h\eta)=\displaystyle\int_0^{\pi}w_{\alpha}(x)(y+h\eta)^2\;dx,}
\end{equation*}
\begin{equation*}
\fonction{\psi}{[-\varepsilon,\varepsilon]}{\R}{h}{J(y+h\eta)=\displaystyle\int_0^{\pi}\left[p(x) ({}^{c}D^{\alpha}_{0+}(y+h\eta))^{2}+q(x)(y+h\eta)^{2}\right]dx}
\end{equation*}
and
\begin{equation*}
\fonction{\zeta}{[-\varepsilon,\varepsilon]}{\R}{h}{R(y+h\eta)=\frac{J(y+h\eta)}{I(y+h\eta)}.}
\end{equation*}
Since $\zeta$ is of class $C^1$ on $[-\varepsilon,\varepsilon]$ and
$$
\zeta(0)\leq\zeta (h),~~\left|h\right|\leq\varepsilon,
$$
we deduce that
\begin{equation*}
\zeta'(0)=\left.\frac{d}{dh}R(y+h\eta)\right|_{h=0}=0.
\end{equation*}
Moreover, notice that
\begin{equation*}
\zeta'(h)=\frac{1}{\phi(h)}\left(\psi'(h)-\frac{\psi(h)}{\phi(h)}\phi'(h)\right)
\end{equation*}
and that
\begin{equation*}
\psi'(0)=\left.\frac{d}{dh}J(y+h\eta)\right|_{h=0}=2\int_0^{\pi}\left[p(x)\cdot {}^{c}D^{\alpha}_{0+}y\cdot{}^{c}D^{\alpha}_{0+}\eta+q(x)y\eta\right]\;dx,
\end{equation*}
\begin{equation*}
\phi'(0)=\left.\frac{d}{dh}I(y+h\eta)\right|_{h=0}=2\int_0^{\pi}\left[w_{\alpha}(x)y\eta\right]\;dx.
\end{equation*}
Therefore
\begin{multline*}
\zeta'(0)=\left.\frac{d}{dh}R(y+h\eta)\right|_{h=0}\\
=\frac{2}{I(y)}\left(\int_0^{\pi}\left[p(x)\cdot {}^{c}D^{\alpha}_{0+}y\cdot{}^{c}D^{\alpha}_{0+}\eta+q(x)y\eta\right]\;dx-\frac{J(y)}{I(y)}\int_0^{\pi}\left[w_{\alpha}(x)y\eta\right]\;dx\right)=0.
\end{multline*}
Having in mind that $\frac{J(y)}{I(y)}=\lambda$ and $\eta(0)=\eta(\pi)=0$, using the integration by parts formula \eqref{eq:IBP} we obtain
\begin{equation*}
\int_0^{\pi}\left(\left[D^{\alpha}_{\pi-}p(x){}^{c}D^{\alpha}_{0+}+q(x)\right]y(x)-\lambda w_{\alpha}(x)y(x)\right)\eta(x)\;dx=0.
\end{equation*}
Now, applying the fundamental lemma of the calculus of variations we arrive at
\begin{equation}\label{aa}
\left[D^{\alpha}_{\pi-}p(x){}^{c}D^{\alpha}_{0+}+q(x)\right]y(x)=\lambda w_{\alpha}(x)y(x).
\end{equation}
Under our assumptions $p(x){}^{c}D^{\alpha}_{0+}y(x)\left|_{x=\pi}\right.=0$ and therefore equation \eqref{aa} is equivalent to
\begin{equation}\label{eq:aa}
\left[{}^{c}D^{\alpha}_{\pi-}p(x){}^{c}D^{\alpha}_{0+}+q(x)\right]y(x)= \lambda w_{\alpha}(x)y(x).
\end{equation}
Since $y\neq 0$ we deduce that number $\lambda$ is an eigenvalue of \eqref{eq:aa}.
On the other hand, let $\lambda^{(m)}$ be an eigenvalue and $y^{(m)}$ the corresponding eigenfunction, then
\begin{equation}\label{eq:SLRa}
\left[{}^{c}D^{\alpha}_{\pi-}p(x){}^{c}D^{\alpha}_{0+}+q(x)\right]y^{(m)}(x)=\lambda^{(m)} w_{\alpha}(x)y^{(m)}(x).
\end{equation}
Similarly to the proof of Theorem~\ref{thm:FE}, we can obtain
\begin{equation*}
\frac{\int_0^{\pi}\left(\left(\Dl y^{(m)}\right)^2 p(x)+q(x)(y^{(m)})^2 \right)\;dx}{\int_0^{\pi}\lambda^{(m)} w_{\alpha}(x)(y^{(m)})^2\;dx}=\lambda^{(m)},
\end{equation*}
for any $m\in\N$. That is $R(y^{(m)})=\frac{J(y^{(m)})}{I(y^{(m)})}=\lambda^{(m)}$. Finally, since minimum value of $R$ at $y$ is equal to $\lambda$, i.e.,
\begin{equation*}
\lambda\leq R(y^{(m)})=\lambda^{(m)}~~\forall m\in\N,
\end{equation*}
we have $\lambda=\lambda^{(1)}$.
\end{proof}

%----------------------------------------------

\subsection{An Illustrative Example}
\label{Example}
Let us consider the following fractional oscillator equation
\begin{equation}\label{eq:example}
\left[p\cdot{}^{c}D^{\alpha}_{b-}{}^{c}D^{\alpha}_{a+}-\lambda\right]y(x)=0,
\end{equation}
where $y(a)=y(b)=0$ and parameter $p>0$. One can easily check that problem of finding nontrivial solutions to equation \eqref{eq:example} and corresponding values of parameter $\lambda$ is a particular case of problem \eqref{eq:SLE}--\eqref{eq:BC} with $p(x)\equiv p$, $q(x)\equiv 0$ and $w_{\alpha}(x)\equiv 1$. The corresponding
minimized functional is
\begin{equation*}
J_{\alpha}(y)= \int_{a}^{b}p\cdot  ({}^{c}D^{\alpha}_{a+}y)^{2}dx= ||\sqrt{p}\;\;{}^{c}D^{\alpha}_{a+}y||^{2}_{L^{2}}
\end{equation*}
with the isoperimetric condition
\begin{equation*}
\int_{a}^{b}y^{2}(x)dx
=1.
\end{equation*}
Let us fix the value of parameter $p$ and  assume that orders $\alpha_{1}, \alpha_{2}$ fulfill the condition: $\frac{1}{2}<\alpha_{1}<\alpha_{2}<1$. Then, we obtain for functionals $J_{\alpha_{1}}, J_{\alpha_{2}}$ the following relation

\begin{eqnarray*}
&& J_{\alpha_{1}}(y)= ||\sqrt{p}{}^{c}D^{\alpha_{1}}_{a+}y||^{2}_{L^{2}}= ||\sqrt{p}I^{1-\alpha_{1}}_{a+}Dy||^{2}_{L^{2}}=||\sqrt{p}I^{\alpha_2-\alpha_1}_{a+}I^{1-\alpha_2}_{a+}Dy||^{2}_{L^{2}}\\
&& \leq K_{\alpha_{2}-\alpha_{1}}^{2}\cdot ||\sqrt{p}{}^{c}D^{\alpha_2}_{a+}y||^{2}_{L^{2}}=K_{\alpha_{2}-\alpha_{1}}^{2}J_{\alpha_{2}}(y),\nonumber
\end{eqnarray*}

where we denoted
\begin{equation*}
K_{\alpha_{2}-\alpha_{1}}:= \frac{(b-a)^{\alpha_2-\alpha_1}}{\Gamma(\alpha_2-\alpha_1+1)}.
\end{equation*}
We observe that in the above estimation two cases occur
\begin{eqnarray*}
& K_{\alpha_{2}-\alpha_{1}}\leq 1 & J_{\alpha_{1}}(y) \leq J_{\alpha_{2}}(y)\\
& K_{\alpha_{2}-\alpha_{1}}>1 & J_{\alpha_{1}}(y) \leq K^{2}_{\alpha_{2}-\alpha_{1}}\cdot J_{\alpha_{2}}(y).
\end{eqnarray*}
The relations between functionals for different values of fractional order
lead to the set of inequalities for eigenvalues $\lambda^{(j)}$ valid for any  $j\in \mathbb{N}$:
\begin{eqnarray*}
& K_{\alpha_{2}-\alpha_{1}}\leq 1 & \lambda^{(j)}(\alpha_1)\leq \lambda^{(j)}(\alpha_2)\\
& K_{\alpha_{2}-\alpha_{1}}>1 & \lambda^{(j)}(\alpha_1)\leq K_{\alpha_{2}-\alpha_{1}}^{2} \cdot\lambda^{(j)}(\alpha_2).
\end{eqnarray*}
In particular, when order $\alpha_{2}=1$ we get

\begin{eqnarray*}
&&J_{\alpha_{1}}(y)=||\sqrt{p}{}^{c}D^{\alpha_{1}}_{a+}y||^{2}_{L^{2}}= ||\sqrt{p}I^{1-\alpha_{1}}_{a+}Dy||^{2}_{L^{2}} \\
&& \leq  K^{2}_{1-\alpha_{1}}\cdot ||\sqrt{p}Dy||^{2}_{L^{2}}=K^{2}_{1-\alpha_{1}}J_{1}(y) \nonumber
\end{eqnarray*}

and the following relations dependent on the value of constant $K_{1-\alpha_{1}}$
\begin{eqnarray*}
& K_{1-\alpha_{1}}\leq 1 & J_{\alpha_{1}}(y) \leq J_{1}(y)\\
& K_{1-\alpha_{1}}>1 & J_{\alpha_{1}}(y) \leq K^{2}_{1-\alpha_{1}}\cdot J_{1}(y).
\end{eqnarray*}
Thus comparing the eigenvalues for the fractional and the classical harmonic oscillator equation for boundary conditions $y(a)=y(b)=0$ we conclude that the respective classical eigenvalues are
higher than the ones resulting from the fractional problem for any  $j\in \mathbb{N}$, namely
\begin{eqnarray}\label{evalu}
& K_{1-\alpha_{1}}\leq 1 & \lambda^{(j)}(\alpha_1)\leq \lambda^{(j)}(1)=p\left (\frac{j\pi}{b-a}\right)^{2}\\
& K_{1-\alpha_{1}}>1 & \lambda^{(j)}(\alpha_1)\leq K^{2}_{1-\alpha_{1}} \cdot\lambda^{(j)}(1)=p\left (\frac{j\pi}{(b-a)^{\alpha_{1}}\Gamma(2-\alpha_{1})}\right)^{2}.
\end{eqnarray}

%----------------------------------------------
\section{Appendix}
\label{sec:ap}

 We shall prove two lemmas, concerning certain convergence properties of fractional and classical derivatives, that play an important role  in the proof of Theorem~\ref{thm:exist}. Let us begin with the following definition of H\"{o}lder continuous functions.

\begin{definition}
Function $g$ is H\"{o}lder continuous in interval $[a,b]$ with coefficient $0<\beta\leq 1$ iff
\begin{equation}
\sup_{x,y\in [a,b], \; x\neq y}\frac{|f(x)-f(y)|}{|x-y|^{\beta}}<\infty.
\end{equation}
We denote this class of H\"{o}lder continuous functions as $C^{\beta}_{H}[a,b]$.
\end{definition}

\begin{lemma}\label{lem:D}
Let  $\alpha\in (0,1) $, functions $w,g\in C^{1}[0,\pi] \cap C^{1}_{H}[-\pi,\pi]$ be odd functions in  $ [-\pi,\pi]$ such that $w'',g''\in L^{2}[0,\pi]$. If we denote as $g_{m}$ the $m$-th sum of the Fourier series
of function $g$, then  the following
convergences are valid in $[0,\pi]$
\begin{eqnarray}
&& \lim_{m\longrightarrow \infty} ||{}^{c}D^{\alpha}_{0,+}g_{m}-{}^{c}D^{\alpha}_{0,+}g||_{L^{2}}=0 \label{con1} \\
&& \lim_{m\longrightarrow \infty} ||D{}^{c}D^{\alpha}_{0,+}g_{m}-D{}^{c}D^{\alpha}_{0,+}g||_{L^{1}}=0\label{con2} \\
&& \lim_{m\longrightarrow \infty} ||{}^{c}D^{\alpha}_{0,+}w g_{m}-{}^{c}D^{\alpha}_{0,+}w g||_{L^{2}}=0 \label{con3}\\
&& \lim_{m\longrightarrow \infty} ||D{}^{c}D^{\alpha}_{0,+}w g_{m}-D{}^{c}D^{\alpha}_{0,+}w g||_{L^{1}}=0.\label{con4}
\end{eqnarray}
\end{lemma}

\begin{proof}
We can apply equation \eqref{K} and  estimate the $||{}^{c}D^{\alpha}_{0,+}g_{m}-{}^{c}D^{\alpha}_{0,+}g||_{L^{2}} $ norm in $[0,\pi]$ as follows
$$ ||{}^{c}D^{\alpha}_{0,+}g_{m}-{}^{c}D^{\alpha}_{0,+}g||_{L^{2}} =  ||I^{1-\alpha}_{0,+}(g'_{m}-g')||_{L^{2}}\leq K_{1-\alpha} \cdot  ||g'_{m}-g'||_{L^{2}}.$$
For odd  functions from the $C^{1}[0,\pi]\cap C^{1}_{H}[-\pi,\pi]$ space,  $g'_{m}$ is the $m$-th sum of the Fourier series of the derivative $g'$. Hence
in interval $[0,\pi]$
$$\lim_{m\longrightarrow \infty} ||g'_{m}-g'||_{L^{2}}=0$$
and from the above inequalities it follows that (\ref{con1}) is valid in $[0,\pi]$
$$ \lim_{m\longrightarrow \infty}||{}^{c}D^{\alpha}_{0,+}g_{m}-{}^{c}D^{\alpha}_{0,+}g||_{L^{2}}=0.$$
Let us observe that for $x>0$
$$ D{}^{c}D^{\alpha}_{0,+}g_{m} (x)= D^{\alpha}_{0,+}g'_{m}(x)={}^{c}D^{\alpha}_{0,+}g'_{m}(x)+\frac{g_{m}'(0)\cdot x^{-\alpha}}{\Gamma(1-\alpha)}$$
$$ D{}^{c}D^{\alpha}_{0,+}g (x)= D^{\alpha}_{0,+}g'(x)={}^{c}D^{\alpha}_{0,+}g'(x)+\frac{g'(0)\cdot x^{-\alpha}}{\Gamma(1-\alpha)}.$$
Therefore  we can  estimate the distance between $D{}^{c}D^{\alpha}_{0,+}g_{m}$ and $D{}^{c}D^{\alpha}_{0,+}g $ in interval $[0,\pi]$
using (\ref{K}) for $\beta=1-\alpha$ and $p=1$
$$ ||D{}^{c}D^{\alpha}_{0,+}g_{m}-D{}^{c}D^{\alpha}_{0,+}g||_{L^{1}}\leq $$
$$ \leq ||{}^{c}D^{\alpha}_{0,+}(g'_{m}-g')||_{L^{1}}+ ||(g'_{m}(0)-g'(0))\cdot \frac{x^{-\alpha}}{\Gamma(1-\alpha)}||_{L^{1}}=$$
$$  = ||I^{1-\alpha}_{0,+}(g''_{m}-g'')||_{L^{1}}+|g'_{m}(0)-g'(0)|\cdot ||\frac{x^{-\alpha}}{\Gamma(1-\alpha)}||_{L^{1}}\leq $$
$$ \leq K_{1-\alpha}\cdot ||g''_{m}-g''||_{L^{1}}+|g'_{m}(0)-g'(0)|\cdot \frac{\pi^{1-\alpha}}{\Gamma(2-\alpha)}\leq $$
$$ \leq K_{1-\alpha}\cdot \sqrt{\pi}\cdot ||g''_{m}-g''||_{L^{2}}+|g'_{m}(0)-g'(0)|\cdot \frac{\pi^{1-\alpha}}{\Gamma(2-\alpha)}.$$
By assumptions we have in $[-\pi,\pi]$ (thence also in $[0,\pi]$)
$$ \lim_{m\longrightarrow \infty}||g''_{m}-g''||_{L^{2}}=0\quad \lim_{m\longrightarrow \infty}|g'_{m}(0)-g'(0)|=0.$$
Hence we conclude that (\ref{con2}) is valid. \\
The convergence given in (\ref{con3}) follows from (\ref{con1}), namely
$$  ||{}^{c}D^{\alpha}_{0,+}w g_{m}-{}^{c}D^{\alpha}_{0,+}w g||_{L^{2}} = $$
$$= ||I^{1-\alpha}_{0,+}\left[\left (w g_{m}\right)'-\left (w g\right)'\right]||_{L^{2}} \leq$$
$$\leq ||I^{1-\alpha}_{0,+}w(g'_{m}-g')||_{L^{2}} +||I^{1-\alpha}_{0,+}\left (w\right)'(g_{m}-g)||_{L^{2}}\leq $$
$$ \leq K_{1-\alpha}\cdot ||w(g'_{m}-g')||_{L^{2}} +K_{1-\alpha}\cdot ||\left (w\right)'(g_{m}-g)||_{L^{2}}\leq $$
$$ \leq K_{1-\alpha} \left (||w||\cdot ||g'_{m}-g'||_{L^{2}} +||\left (w\right)'|| \cdot ||g_{m}-g||_{L^{2}}\right),$$
where $||\cdot||$ denotes the supremum norm in the $C[0,\pi]$ space. From assumptions of our lemma it follows that in $[0,\pi]$
$$\lim_{m\longrightarrow \infty}||g'_{m}-g'||_{L^{2}}=0\quad \lim_{m\longrightarrow \infty}||g_{m}-g||_{L^{2}}=0.$$
Thus convergence (\ref{con3}) is valid.\\
To prove convergence (\ref{con4}) we start by observing that for $x>0$
$$ D{}^{c}D^{\alpha}_{0,+}w g_{m} (x)= D^{\alpha}_{0,+}\left(w g_{m}\right)'(x)={}^{c}D^{\alpha}_{0,+}\left(w g_{m}\right)'(x)+\frac{(w g_{m})'(0)\cdot x^{-\alpha}}{\Gamma(1-\alpha)}$$
$$ D{}^{c}D^{\alpha}_{0,+}w g (x)= D^{\alpha}_{0,+}\left(w g\right)'(x)={}^{c}D^{\alpha}_{0,+}\left (w g\right)'(x)+\frac{(w g)'(0)\cdot x^{-\alpha}}{\Gamma(1-\alpha)}.$$
For the $L^{1}$-distance between $D{}^{c}D^{\alpha}_{0,+}w g_{m}$ and $D{}^{c}D^{\alpha}_{0,+}w g $ in interval $[0,\pi]$
we have
\begin{equation}
 ||D{}^{c}D^{\alpha}_{0,+}w g_{m}-D{}^{c}D^{\alpha}_{0,+}w g||_{L^{1}}\leq
\label{est}
\end{equation}
$$ \leq ||{}^{c}D^{\alpha}_{0,+}\left[\left (w g_{m}\right)'-\left (w g\right)'\right]||_{L^{1}}+ ||\left [\left(w g_{m}\right)'(0)-\left(w g\right)'(0)\right]\cdot \frac{x^{-\alpha}}{\Gamma(1-\alpha)}||_{L^{1}}=$$
$$ \leq ||I^{1-\alpha}_{0,+}\left[\left (w g_{m}\right)''-\left (w g\right)''\right]||_{L^{1}}+ \left|\left(w g_{m}\right)'(0)-\left(w g\right)'(0)\right|\cdot ||\frac{x^{-\alpha}}{\Gamma(1-\alpha)}||_{L^{1}}\leq$$
$$ \leq K_{1-\alpha}\cdot || \left (w g_{m}\right)''-\left (w g\right)''||_{L^{1}}+\left|\left(w g_{m}\right)'(0)-\left(w g\right)'(0)\right|\cdot \frac{\pi^{1-\alpha}}{\Gamma(2-\alpha)}\leq $$
$$ \leq K_{1-\alpha}\cdot \sqrt{\pi}\cdot || \left (w g_{m}\right)''-\left (w g\right)''||_{L^{2}}+\left|\left(w g_{m}\right)'(0)-\left(w g\right)'(0) \right|\cdot \frac{\pi^{1-\alpha}}{\Gamma(2-\alpha)}.$$
Because
$$\left (w g_{m}\right)''-\left (w g\right)'' = $$
$$ = w(g''_{m}-g'') + 2 \left (w\right)'\cdot (g'_{m}-g') +\left (w\right)''\cdot (g_{m}-g)$$
we have
$$ || \left (w g_{m}\right)''-\left (w g\right)''||_{L^{2}}\leq$$
$$ \leq ||w||\cdot || g''_{m} - g''||_{L^{2}} + 2 \cdot ||\left (w\right)'||\cdot || g'_{m} - g'||_{L^{2}}+ ||\left (w\right)''||_{L^{2}}\cdot || g_{m} - g||_{L^{2}}.$$
From the assumptions of the lemma it follows that for $j=0,1,2$
$$ \lim_{m\longrightarrow \infty}|| g^{(j)}_{m} - g^{(j)}||_{L^{2}} =0.$$
Hence
$$ \lim_{m\longrightarrow \infty}|| \left (w g_{m}\right)''-\left (w g\right)''||_{L^{2}}=0.$$
In addition
$$ \lim_{m\longrightarrow \infty}\left|\left (w g_{m}\right)'(0)-\left(w g\right)'(0)\right| =$$
$$ = \lim_{m\longrightarrow \infty}\left|\left (w\right)'(0)(g_{m}(0)-g(0)) + w(0)(g'_{m}(0)-g'(0))\right|\leq $$
$$ \leq \lim_{m\longrightarrow \infty}\left|\left (w\right)'(0)(g_{m}(0)-g(0))\right| +  \lim_{m\longrightarrow \infty}\left|w(0)(g'_{m}(0)-g'(0))\right|=0.$$
Taking into account estimation (\ref{est}) and the above inequalities we conclude that (\ref{con4}) is valid.
\end{proof}

\begin{lemma}\label{lem:E}
Let $\alpha\in \left(\frac{1}{2},1\right), \;\; \beta \leq \alpha -\frac{1}{2}  $, function $w$  be positive, even function in  $ [-\pi,\pi]$ and  $w'\in  C^{\beta}_{H}[-\pi,\pi] $. Function $h'$ is the derivative of $h$  defined by assumptions of Lemma \ref{lem:B} and formula \eqref{defh2}, function $g$ is defined as
$$g(x):= h(x)w(x).$$
If we denote as $g_{m}$ the $m$-th sum of the Fourier series
of function $g$, then  the following
convergences are valid in interval $[0,\pi]$
\begin{eqnarray}
&& \lim_{m\longrightarrow \infty}||g'_{m}-g'||=0   \label{cp1}\\
&& \lim_{m\longrightarrow \infty} g'_{m}(0)=g'(0) \label{cp2}\\
&&\lim_{m\longrightarrow \infty} g'_{m}(\pi)=g'(\pi).\label{cp3}
\end{eqnarray}
\end{lemma}

\begin{proof}
Definition (\ref{defh2})
in interval $[0,\pi]$ implies for derivative $h'$
\begin{equation}
h'(x)= I^{\alpha}_{0+}\gamma(x)+Ax^{\alpha}+Bx^{1+\alpha},
\end{equation}
where $\gamma\in C[0,\pi]$ and constants $A,B\in \mathbb{R}$ are specified by  conditions \eqref{cn1}, \eqref{cn2} in the proof of Lemma 1. Let us observe that $x^{1+\alpha}\in C^{1}[0,\pi]$, function 
$x^{\alpha}$ is H\"{o}lder continuous  in $[0,\pi]$ with coefficient $\beta \leq \alpha$, thus it can be extended to an odd/even,
 H\"{o}lder continuous function in interval $[-\pi,\pi]$.  In addition $I^{\alpha}_{0+}\gamma(x)$ is H\"{o}lder continuous function in $[0,\pi]$ with coefficient $\beta \leq \alpha-\frac{1}{2}$
because:
$$\frac{|I^{\alpha}_{0+}\gamma(x)-I^{\alpha}_{0+}\gamma(y)|}{|x-y|^{\beta}} \leq \frac{2 \cdot ||\gamma||_{L^{2}}}{\Gamma(\alpha)\sqrt{2\alpha-1}}\cdot |x-y|^{\alpha-\frac{1}{2}-\beta}\leq \frac{2 \cdot ||\gamma||_{L^{2}}}{\Gamma(\alpha)\sqrt{2\alpha-1}}\cdot \pi^{\alpha-\frac{1}{2}-\beta}<\infty$$
and can be extended  to an odd/even,
 H\"{o}lder continuous function in interval $[-\pi,\pi]$.
Observe that for H\"{o}lder continuous functions in $[-\pi,\pi]$ we have the absolute convergence of their
Fourier series.
For function $g' $ we obtain in $[0,\pi]$
$$ g'(x) = h'(x)w(x)+h(x)w'(x).$$
Both  terms on the right-hand side  are by assumption  functions from the $C^{\beta}_{H}[0,\pi]$-space and  can be  extended to  odd/even functions
in the $C^{\beta}_{H}[-\pi,\pi]$ space.  Hence their Fourier series are absolutely convergent in $[-\pi,\pi]$.
Concluding, we have for function $g'$ the convergence in interval $[-\pi,\pi]$
$$ \lim_{m\longrightarrow \infty}||g'_{m}-g'||=0,$$
where $||\cdot || $ denotes the supremum norm
in interval $[-\pi,\pi]$. Thus the sequence  $g'_{m}$ of partial sums is
also absolutely convergent in interval $[0,\pi]$.
Formulas (\ref{cp2},\ref{cp3}) are a  straightforward consequence of this fact.
\end{proof}

% -----------------------------------------------------------------------------

\section*{Acknowledgements}
Research supported under Czestochowa University of Technology project BS/PB-1-105-3010/2011/S (M. Klimek), 
Bialystok
University of Technology grant S/WI/02/2011 (A.B. Malinowska)
and by {\it Center for Research and Development
in Mathematics and Applications}  (University of Aveiro)
and the Portuguese Foundation for Science and Technology
(``FCT --- Funda\c{c}\~{a}o para a Ci\^{e}ncia e a Tecnologia''),
within project PEst-C/MAT/UI4106/2011
with COMPETE number FCOMP-01-0124-FEDER-022690 (T. Odzijewicz).

%------------------------------------------------------------------------------------

\end{document}